\documentclass[10 pt, conference]{ieeeconf}  

\IEEEoverridecommandlockouts                              

\usepackage{graphics} 
\usepackage{epsfig} 
\usepackage{amsmath} 

\usepackage{amssymb}
\usepackage{amsthm}  
\usepackage{amscd}
\usepackage{dsfont}
\usepackage{cite}
\usepackage{float}
\usepackage{mdframed}

\usepackage{times}
\usepackage{xcolor}
\usepackage{comment}
\newtheorem{theorem}{Theorem}

\newtheorem{lemma}{Lemma}

\newtheorem{proposition}{Proposition}

\newtheorem{assumption}{Assumption}

\newcommand{\beq}{\begin{equation}}
\newcommand{\eeq}{\end{equation}}
\newcommand{\beqa}{\begin{eqnarray}}
\newcommand{\eeqa}{\end{eqnarray}}
\newcommand{\beqan}{\begin{eqnarray*}}
\newcommand{\eeqan}{\end{eqnarray*}}

\newcommand{\Lcal}{{\cal L}}

\newcommand{\Wcal}{{\cal W}}

\newcounter{l1}
\newcounter{l2}
\newcounter{l3}
\newcommand{\bdotlist}{\begin{list}{$\bullet$}{}}
\newcommand{\bboxlist}{\begin{list}{$\Box$}{}}
\newcommand{\bbboxlist}{\begin{list}{\raisebox{.005in}{{\tiny
$\blacksquare$ \ \ }}}{}}
\newcommand{\bdashlist}{\begin{list}{$-$}{} }
\newcommand{\blist}{\begin{list}{}{} }
\newcommand{\barablist}{\begin{list}{\arabic{l1}}{\usecounter{l1}}}
\newcommand{\balphlist}{\begin{list}{(\alph{l2})}{\usecounter{l2}}}
\newcommand{\bAlphlist}{\begin{list}{\Alph{l2}.}{\usecounter{l2}}}
\newcommand{\bdiamlist}{\begin{list}{$\diamond$}{}}
\newcommand{\bromalist}{\begin{list}{(\roman{l3})}{\usecounter{l3}}}

\title{\LARGE \bf
Quantification of Market Power Mitigation via Efficient Aggregation of Distributed Energy Resources
}
\author{Zuguang Gao, Khaled Alshehri, and John R. Birge
\thanks{Z.~Gao and J.~R.~Birge are with Booth School of Business, The University of Chicago, Chicago, IL 60637, USA. 
        {\tt\small \{zuguang.gao, john.birge\}@chicagobooth.edu}}%
\thanks{K. Alshehri is with the Control and Instrumentation Engineering Department and the Interdisciplinary Research Center for Smart Mobility and Logistics, King Fahd University of Petroleum and Minerals (KFUPM), Dhahran, Saudi Arabia.
        {\tt\small kalshehri@kfupm.edu.sa}}  
\thanks{This work was supported in part by the National Science Foundation (NSF) Award No.~1832230. Z.~Gao and J.~R.~Birge would like to acknowledge the
	support from the University of Chicago Booth School of Business. The work of K.~Alshehri was supported by the Interdisciplinary
    	Research Center for Smart Mobility and Logistics at KFUPM under Grant No.~INML2106.}    
}
\begin{document}
\maketitle
\thispagestyle{empty}
\pagestyle{empty}

\begin{abstract}
	
Distributed energy resources (DERs) such as solar panels have small supply capacities and cannot be directly integrated into wholesale markets. So, the presence of an intermediary is critical. The intermediary could be a profit-seeking entity (called the aggregator) that buys DER supply from prosumers, and then sells them in the wholesale electricity market. Thus, DER integration has an influence on wholesale market prices, demand, and supply. The purpose of this article is to shed light onto the impact of efficient DER aggregation on the market power of conventional generators. Firstly, under efficient DER aggregation, we quantify the social welfare gap between two cases: when conventional generators are truthful, and when they are strategic. We also do the same when DERs are not present. Secondly, we show that the gap due to market power of generators in the presence of DERs is smaller than the one when there is no DER participation. Finally, we provide explicit expressions of the gaps and conduct numerical experiments to gain deeper insights.  The main message of this article is that market power of conventional generators can be mitigated by adopting an efficient DER aggregation model.


\end{abstract}

\section{INTRODUCTION}

Distributed energy resources (DERs), according to the North American Electric Reliability Corporation (NERC), are defined as ``any resource on the distribution system that produces electricity and is not otherwise included in the formal NERC definition of the Bulk Electric System (BES)"~\cite{derdefinition}. Their increasing integration poses fundamental challenges to the design and operation of electricity markets~\cite{ritzenhofen2016structural}, as they are energy sources for the distribution power system, but at the same time, independent market operators do not have enough visibility over them. This is a concern for market operators because DERs dramatically influence the demand behavior. One common way of integrating DERs into wholesale markets is via an intermediary, called the {\it Aggregator}, which could be a profit-seeking entity or a utility with good visibility over the distribution lines.

In~\cite{9683117}, we proposed an aggregation model with a profit-seeking aggregator which achieves full market efficiency. Specifically, every prosumer's behavior (buying/selling amount), along with the market price and social welfare, are exactly the same as if these prosumers participate directly in the wholesale market, which is the ideal but unrealistic case, as aggregators are necessary intermediaries (see Fig.~\ref{sketch} for an illustration). There are two underlying assumptions in the model: all prosumers bid truthfully about their utility of consumption, and all generators bid truthfully about their cost of production. The former assumption is more justifiable as each prosumer represents a small fraction in the energy market; thus, she does not believe her bidding could make a difference in the market price. The later assumption, however, may not hold in general as the generators' bidding might affect the market price, and they may benefit from non-truthful (strategic) bidding. Such strategic bidding may result in a higher market price (thus benefiting the generators) and reduced social welfare. 
\begin{figure}
	\centering
	\includegraphics[width=\linewidth]{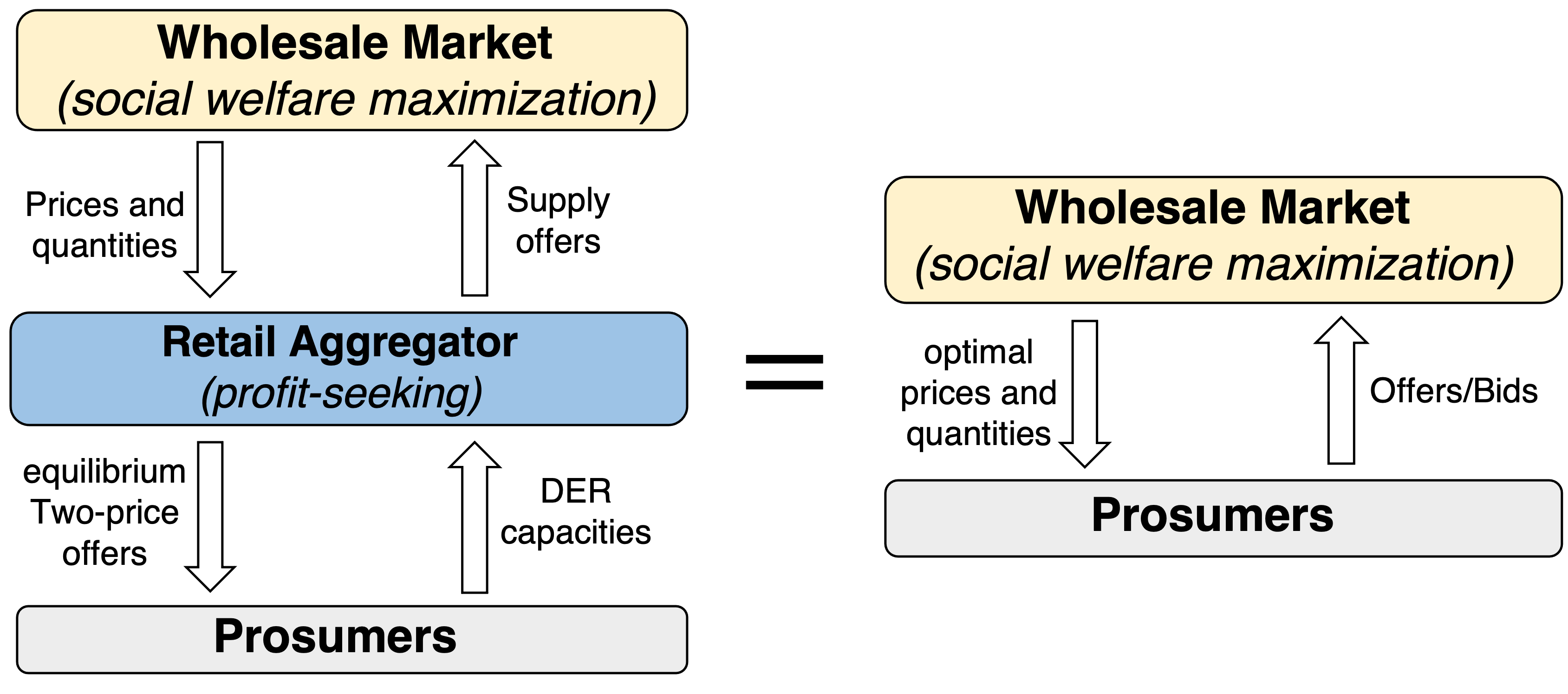}
	\caption{Via the efficient DER aggregation proposed in \cite{9683117}, direct DER participation (benchmark case) is equivalent to having a profit-seeking monopolistic aggregator. }
	\label{sketch}
\end{figure}

The ability of the generators to influence the market price is referred to as the \emph{market power} of the generators, which could negatively affect the social welfare~\cite{al2017understanding}. Understanding market power of generators has been, and still is, an active area of research \cite{marketpower1,marketpower2,marketpower3,marketpower4,marketpower5}. While some articles focus on country-specific issues \cite{marketpower2,marketpower4,marketpower6,marketpower7,marketpower8,marketpower9,marketpower10,marketpower11}, it is evident that understanding market power and potential price manipulations is of strong interest, especially with DERs being intermittent and uncontrollable \cite{DERpower,martinez2016impact,AlshehriBoseBasarIJEPES, cai2020inefficiency}. As DER adoption increases rapidly, it becomes critical to understand the effect of DER integration on the market power of conventional generators, which is, to the best of our knowledge, an issue with plenty of open questions to address.  Hence, in this paper, we study the {market power} of the generators when there is no DER aggregation (no prosumer participation) and when prosumers fully participate (either directly or via efficient aggregation) in the wholesale market. Specifically, we address the following question: {\em When DERs are aggregated, can the market power of conventional power generators be mitigated? If yes, to what extent?}  We will show (both qualitatively and quantitatively) that, compared to no prosumer participation, the market power of generators is mitigated, and we will quantify this reduction in terms of social welfare. 
In particular, we prove that under strategic bidding, the social welfare is higher under efficient aggregation, compared to the case when DERs are not integrated. We also prove that the welfare gap between truthful bidding and strategic bidding is smaller when DERs are aggregated, compared to the case where there are no DERs. Finally, we provide quantifications of such differences.

The paper is organized as follows: Section~\ref{sec:pre} discusses the motivation and summarizes the main result of the paper, which we will gradually prove throughout the paper. In Section~\ref{sec:full}, we quantify the equilibrium quantities of the efficient DER aggregation model (or, equivalently, the benchmark case of direct prosumer participation in wholesale markets),  and then, in Section~\ref{sec:no}, we do the same for the for the case in which there is no DER participation.  We offer a theoretical discussion along with explicit quantification of market power mitigation from efficient DER aggregation  in Section~\ref{sec:dis}. Then, we conduct numerical experiments in Section~\ref{sec:ex}, where we provide deeper insights. The paper concludes in Section~\ref{sec:con} with a summary and suggestions for future research. 
Key proofs are provided in the Appendix, while others are omitted due to space limits, but can be found in~\cite{gao2021aggregating}. 



\section{MOTIVATION AND MAIN RESULT}\label{sec:pre}

Suppose there are~$n$ prosumers, indexed by~$i$, and~$N$ generators, indexed by~$j$.  For ease of exposition, we restrict the attention to one node (location).
Prosumer~$i$ is equipped with capacity~$C_i$ of electricity generation, and makes a decision~$z_i$, which is the \emph{net} amount of energy bought, to maximize her total payoff. Generator~$j$ chooses the production amount~$y_j$, which incurs a cost~$c_j(y_j)$, to maximizes its total profit. We also make the following assumptions on prosumers' utility of consumption and generators' cost functions.

\begin{assumption}\label{assum:power}
	For each prosumer~$i$, the utility of consumption is quadratic and it is given by $u_i(C_i+z_i) = a_i(C_i+z_i)^2+b_i(C_i+z_i)$, where $a_i<0$ and $b_i\gg |a_i|$ so that the market price always falls into the range $\left(0, b_i\right)$. Each generator~$j$'s (true) cost function is linear in its production, and all generators have the same cost function, i.e., $c_j(y_j) = \alpha y_j$ for some $\alpha>0$, and the optimal total supply always satisfies~$y>0$. Furthermore, there exists at least one prosumer~$i$ such that $2a_iC_i+b_i> \alpha$.
\end{assumption}

As we will see in the rest of this paper, Assumption~\ref{assum:power} enables us to obtain explicit expressions for the generators' supply and the social welfare, so that we can characterize the generators' market power quantitatively. In the following, we denote by~$\Wcal^T$ (resp., $\Wcal^S$) the optimal social welfare of the model with prosumer participation when generators bid truthfully (resp., strategically). When no prosumer participation is allowed, the optimal social welfares under truthful bidding and strategic bidding of the generators are denoted by $\Wcal^{TN}$ and $\Wcal^{SN}$, respectively. Our first main result in this paper is summarized as the following theorem.

\begin{theorem}[Main Result]\label{thm:power}
	Under Assumption~\ref{assum:power}, the following inequalities hold:
	\begin{subequations}
		\begin{align}
			&\Wcal^T  \ge \Wcal^{TN}, \ \Wcal^S \ge \Wcal^{SN}, \label{eq:thm1a}\\ 
			&\Wcal^T \ge \Wcal^{S}, \ \Wcal^{TN} \ge \Wcal^{SN},\label{eq:thm1b}\\
			&\Wcal^{TN} - \Wcal^{SN} \ge \Wcal^T - \Wcal^{S}\label{eq:thm1c}
		\end{align}
	\end{subequations}
\end{theorem}
In Theorem~\ref{thm:power}, the first two inequalities~\eqref{eq:thm1a} state that the optimal social welfare with prosumer participation is always greater than that without prosumer participation, for both truthful bidding and strategic bidding of the generators. The next two inequalities~\eqref{eq:thm1b} state that the optimal social welfare under truthful bidding is always greater than that under strategic bidding of generators, for both cases with prosumer participation and without prosumer participation. Finally, the last inequality~\eqref{eq:thm1c} imply that the loss of social welfare due to strategic bidding of the generators are reduced when there is prosumer participation, compared to the case when there is no prosumer participation.

For the rest of this paper, we will provide a complete analysis of Theorem~\ref{thm:power}, along with the explicit analytical expressions for the social welfares and market prices under all four models. 

\section{FULL PROSUMER PARTICIPATION}\label{sec:full}
We first consider the case when there is prosumer participation. As shown in~\cite{9683117}, given a  market price, the prosumers' decisions, as well as the social welfare, under the aggregation model are the same as those as if they participate directly in the wholesale market. Therefore, we will without loss of generality ignore the role of the aggregator and assume the direct participation of prosumers. We analyze the decisions of each party for the cases when generators bid truthfully and strategically.

\subsection{Truthful bidding of the generators}

\subsubsection{Prosumers} Let $\lambda^T$ be the market price when generators bid truthfully. Each prosumer~$i$ solves her payoff maximization problem:
\small
\begin{align}
	\max_{z_i>-C_i} u_i(C_i+z_i) - \lambda^T z_i.
\end{align}
\normalsize
Under Assumption~\ref{assum:power}, prosumer~$i$'s optimal response~$z_i^T$ satisfies the first order condition:
$
2a_i(C_i+z_i^T) + b_i = \lambda^T.
$
%
%
\subsubsection{System operator} The system operator solves the economic dispatch problem to maximize the social welfare~$\Wcal^T$:
\small
\begin{equation}\label{eq:sop}
	\begin{aligned}
		\Wcal^T = &\max_{z_i>-C_i, y_j\ge 0}\sum_iu_i(C_i+z_i) - \sum_{j}c_j(y_j)\\
		&\qquad \text{s.t. }\quad  \sum_jy_j = \sum_iz_i.
	\end{aligned}
\end{equation}

\normalsize
Since the generators are identical, we may without loss of generality restrict attention to $y_j = y_{j'}, \forall j\ne j'$, i.e., each generator supplies the same amount of energy~$y_j$. Let $y:=\sum_{j}y_j$ and $C:=\sum_iC_i$.
Considering the prosumers' and the system operator's problems, we have the following result.
\begin{proposition}\label{prop:power1}
	Under Assumption~\ref{assum:power}, an optimal solution to~\eqref{eq:sop} is given by
	\small
				\begin{align}
			z_i^T =  - C_i + \frac{\alpha - b_i}{2a_i},\ \forall i,\quad y_j^T = \frac{-C + \sum_{i}\frac{\alpha -b_i}{2a_i}}{N},\ \forall j.
				\end{align}
	\normalsize
	Furthermore, the equilibrium market price (the optimal Lagrange multiplier of the constraint) is given by $\lambda^T = \alpha$.
\end{proposition}

\subsection{Strategic bidding of the generators}
\subsubsection{Prosumers} Let $\lambda^S$ be the market price when generators bid strategically. Then, each prosumer~$i$'s optimal response $z_i^S$ (under Assumption~\ref{assum:power}) now becomes
$
2a_i(C_i+z_i^S) + b_i = \lambda^S.
$

\subsubsection{System operator} The system operator solves the economic dispatch problem to maximize the \emph{apparent social welfare}, which is the ``social welfare" when the system operator assumes the generators' bids are true, but are actually based on the (nontruthful) bidding $\tilde{c}_j$ of the generators:
\small
\begin{equation}\label{eq:sops}
	\begin{aligned}
		&\max_{z_i>-C_i, y_j\ge 0}\sum_iu_i(C_i+z_i) - \sum_{j}\tilde{c}_j(y_j)\\
		&\qquad\text{s.t. }\quad  \sum_jy_j = \sum_iz_i.
	\end{aligned}
\end{equation}
\normalsize

For each possible total supply~$y$, we define the overall utility of consumption as
\small
\begin{align}\label{eq:u}
	u(C+y) = \left\{\max_{z_i>-C_i}\sum_iu_i(C_i+z_i)\ \text{s.t.}\ \sum_iz_i  = y\right\}.
\end{align}
\normalsize
In other words, $u(C+y)$ computes the total utility of consumption of all prosumers when the total energy supply from the generators is~$y$. We can therefore rewrite the system operator's problem equivalently as
\small
\begin{align}
	\max_{y}u(C+y) - \sum_j\tilde{c}_j(y/N).
\end{align}
\normalsize
\subsubsection{Generators}  Each generator sets an optimal~$y_j$ to supply, and bids a tilted cost function~$\tilde{c}_j$ instead of the true cost function~$c_j$. The generator~$j$ considers the market price~$\lambda^S$ as a function of the total supply~$y$, and aims to solve the profit maximization problem:
\small
\begin{align}
	\max_{y_j}\lambda^S\Big(y_j + \sum_{j'\ne j}y_{j'}\Big)\cdot y_j -c_j(y_j).
\end{align}
\normalsize
To this end, the generator needs to compute the market price as a function of total supply. Since the total supply and the total (net) demand are matched by the system operator, we have that
$
\sum_iz_i^S = y.
$
This, together with prosumer's optimal condition $2a_i(C_i+z_i^S) + b_i = \lambda^S$, implies that
\small
\begin{align}
	\lambda^S(y) = \frac{y+C+\sum_i\frac{b_i}{2a_i}}{\sum_i\frac{1}{2a_i}}.
\end{align}
\normalsize
Therefore, the generator's profit maximization problem becomes:
\small
\begin{align}\label{eq:genps1}
	\max_{y_j\ge 0} \frac{y_j+\sum_{j'\ne j}y_{j'}+C+\sum_i\frac{b_i}{2a_i}}{\sum_i\frac{1}{2a_i}}\cdot y_j - \alpha y_j.
\end{align}
\normalsize
Each generator solves~\eqref{eq:genps1}, and since they are all identical, we only look at the equilibrium where $y_j = y_{j'},\forall j,j'\in[N]$. Thus, from the first-order condition of~\eqref{eq:genps1}, we obtain the optimal supply amount for generator~$j$:
\small
\begin{align}\label{eq:yjs}
	y_j^S = \frac{-C+\sum_i\frac{\alpha - b_i}{2a_i}}{N+1},\quad \forall j.
\end{align}
\normalsize
In order to sell~$y_j^S$ as given in~\eqref{eq:yjs}, generator~$j$ will bid~$\tilde{c}_j$ such that the system operator will allocate~$y_j^S$ amount of energy to~$j$. The generator thus considers the economic dispatch problem that the system operator solves.
Therefore, the generator will bid~$\tilde{c}_j$ such that~$y^S:= \sum_jy_j^S = Ny_j^S$ solves the system operator's problem optimally, i.e.,
\small
\begin{align}\label{eq:bidscon}
	\frac{\partial u(C+y)}{\partial y}\bigg|_{y=Ny_j^S} = \frac{\partial \tilde{c}_j(y_j)}{\partial y_j}\bigg|_{y_j = y_j^S}.
\end{align}
\normalsize
While there are many possible choices of~$\tilde{c}_j$ that satisfies~\eqref{eq:bidscon}, one of the optimal bids for the generator is to bid a linear cost function.
\begin{lemma}\label{lem:bids}
	The following linear cost function is an optimal bidding strategy for  generator~$j$:
	\small
	\begin{align}\label{eq:bids}
		\tilde{c}_j(y_j) 
		= \frac{N\alpha\sum_i\frac{1}{2a_i} + C +\sum_i\frac{b_i}{2a_i}}{(N+1)\sum_i\frac{1}{2a_i}}\cdot y_j
	\end{align}
\normalsize
\end{lemma}

Considering the prosumers' and the system operator's problems, we have the following result.
\begin{proposition}[Competitive Equilibrium]\label{prop:power2}
	Under Assumption~\ref{assum:power}, if all generators bid as in~\eqref{eq:bids}, then, an optimal solution to~\eqref{eq:sops} is given by
	\small
				\begin{align}
						z_i^S =  - C_i + \frac{\lambda^S - b_i}{2a_i},\ \forall i,\quad
						y_j^S = \frac{-C + \sum_{i}\frac{\alpha -b_i}{2a_i}}{N+1},\ \forall j,\label{eq:zsysys}
					\end{align}
\normalsize
		where the equilibrium market price (the optimal Lagrange multiplier of the equality constraint) is
		\small
		$$
		\lambda^S
		= \frac{N\alpha}{N+1} + \frac{C+\sum_i\frac{b_i}{2a_i}}{(N+1)\sum_i\frac{1}{2a_i}}.
		$$
		\normalsize
		Furthermore, the solutions in~\eqref{eq:zsysys} optimally solve the prosumer's/generator's problem.


	\end{proposition}
	
	\section{NO PROSUMER PARTICIPATION}\label{sec:no}
	We next consider the case when the prosumers cannot sell back to the grid, i.e., each prosumer~$i$ can only purchase some $z_i\ge 0$ amount of energy. We first look at the prosumers' problem.
	
	\subsubsection*{Prosumers} Let $\lambda$ be the market price. Each prosumer~$i$ solves her payoff maximization problem:
	$
	\max_{z_i\ge 0} u_i(C_i+z_i) - \lambda z_i = \max_{z_i\ge 0}a_i(C_i+z_i)^2+b_i(C_i+z_i) - \lambda z_i.
	$
	The optimal decision of prosumer~$i$ is thus
	\small
	\begin{align}\label{eq:disn}
		z_i = \left[\frac{\lambda-b_i}{2a_i}-C_i\right]^+.
	\end{align}
\normalsize
	Therefore, given any market price~$\lambda$, the set of prosumers who make a strictly positive amount of purchase is
	$
	\mathcal{S}(\lambda) := \left\{i\mid 2a_iC_i+b_i>\lambda\right\}.
	$
	We will without loss of generality sort the prosumers in decreasing order of~$2a_iC_i+b_i$, i.e., for any two prosumers~$i$ and~$i'$, we have that $2a_iC_i+b_i\ge 2a_{i'}C_{i'}+b_{i'}$ if $i\le i'$. Under such ordering, for any $i\le i'$, if prosumer~$i'$ makes a positive purchase of electricity, then prosumer~$i$ must also make a positive purchase. This sorting of prosumers will be helpful when we write the social welfare.
	
	\subsection{Truthful bidding of the generators}
	
	\subsubsection{System operator} The system operator solves the economic dispatch problem that maximizes the social welfare:
	\small
	\begin{equation}\label{eq:sopn}
		\begin{aligned}
			\Wcal^{TN} = &\max_{z_i\ge 0, y_j\ge 0}\sum_iu_i(C_i+z_i) - \sum_{j}c_j(y_j)\\
			&\qquad \text{s.t. }\quad  \sum_jy_j = \sum_iz_i.
		\end{aligned}
	\end{equation}
\normalsize
	Since the generators are identical, we again restrict solutions to $y_j = y_{j'}, \forall j\ne j'$, i.e., each generator supplies the same amount of energy~$y_j$.
	Considering the prosumers' and the system operator's problems, we have the following result.
	\begin{proposition}\label{prop:power1n}
		Under Assumption~\ref{assum:power}, an optimal solution to~\eqref{eq:sopn} is given by
		\small
			\begin{subequations}
					\begin{align}
							z_i^{TN} &=  \left[- C_i + \frac{\alpha - b_i}{2a_i}\right]^+,\quad \forall i,\\
							y_j^{TN} &= \frac{\sum_i\left[-C_i+\frac{\alpha-b_i}{2a_i}\right]^+}{N},\quad \forall j.
						\end{align}
				\end{subequations}
\normalsize
		Furthermore, the equilibrium market price (the optimal Lagrange multiplier of the constraint) is given by
			$\lambda^{TN} = \alpha.$
	\end{proposition}

	\subsection{Strategic bidding of the generators}
	\subsubsection{System operator} The system operator solves the economic dispatch problem to maximize the \emph{apparent social welfare}, which is the ``social welfare" when the system operator assumes the generators' bids are true, but are again actually based on the (nontruthful) bidding $\tilde{c}_j$ of the generators: 
	\small
	\begin{equation}\label{eq:sopns}
		\begin{aligned}
			&\max_{z_i\ge 0, y_j\ge 0}\sum_iu_i(C_i+z_i) - \sum_{j}\tilde{c}_j(y_j)\\
			&\qquad \text{s.t. }\quad  \sum_jy_j = \sum_iz_i
		\end{aligned}
	\end{equation}
	\normalsize
	
	For each possible total supply~$y$, the overall utility of consumption now becomes
	\small
	\begin{align}\label{eq:uc+yn}
		u(C+y) = \left\{\max_{z_i\ge 0}\sum_iu_i(C_i+z_i)\ \text{s.t.}\ \sum_iz_i  = y\right\}.
	\end{align}
\normalsize
	As $y$ increases, the number of prosumers with $z_i >0$ will change in a discrete manner. However, we have the following useful result.
	\begin{lemma}\label{lem:ucondiff}
		$u(C+y)$ is continuous and differentiable in~$y$.
	\end{lemma}
	We can thus rewrite the system operator's problem equivalently as
	\small
	\begin{align}
		\max_{y\ge 0} u(C+y) - \sum_j\tilde{c}_j(y/N).
	\end{align}
\normalsize

	
	\subsubsection{Generators}  Each generator sets an optimal~$y_j$ to supply, and bids a tilted cost function~$\tilde{c}_j$ instead of the true cost function~$c_j$. The generator~$j$ considers the market price~$\lambda^{SN}$ as a function of the total supply~$y$, and aims to solve the profit maximization problem:
	\small
	\begin{align}\label{eq:genpsn}
		\max_{y_j} \lambda^{SN}\Big(y_j + \sum_{j'\ne j}y_{j'}\Big)\cdot y_j -c_j(y_j).
	\end{align}
\normalsize
	To this end, the generator needs to compute the market price as a function of total supply. Since the total supply and the total (net) demand are matched by the system operator, we have that
	$
	\sum_iz_i^{SN} = y.
	$
	At any given~$y$, we need to consider which prosumers have~$z_i>0$ and which prosumers have~$z_i=0$. For all $i=\{1,2,\ldots,n\}$, define
	\small
	\begin{align}
		y^{i} := \left(2a_{i}C_{i}+b_{i}\right)\sum_{i'=1}^{i-1}\frac{1}{2a_{i'}} - \sum_{i'=1}^{i-1}\left(C_{i'}+\frac{b_{i'}}{2a_{i'}}\right).
	\end{align}
\normalsize
	Then, we have the following lemma.
	\begin{lemma}\label{lem:yi}
		Prosumer~$i$'s optimal decision~$z_i>0$ if and only if the total supply satisfies~$y> y^i$.
	\end{lemma}
	In other words, the set of prosumers with~$z_i>0$ may be written as
	$
	\mathcal{S}(y) = \left\{i\mid y> y^i\right\}.
	$
	\begin{figure*}[!ht]
\centering
\small
\begin{mdframed}
\begin{itemize}
\item Full prosumer participation (truthful bidding): $$	\Wcal^T = \alpha C - \sum_i\frac{(b_i-\alpha)^2}{4a_i}$$
\item Full prosumer participation (strategic bidding):$$\Wcal^S = \left(\frac{{\sum_i\left[\frac{\alpha N +b_i}{2a_i}+C_i\right]}}{(N+1)\sum_i\frac{1}{2a_i}}\right)^2 \sum_i \frac{1}{4a_i} -\sum_i \frac{b_i^2}{4a_i} + \alpha \frac{N\sum_i\left(C_i-\frac{\alpha - b_i}{2a_i}\right)}{N+1}$$
\item No prosumer participation (truthful bidding): $$\Wcal^{TN} = \sum_{\left\{i\mid 2a_iC_i+b_i>\alpha\right\}} \left[\alpha C_i - \frac{(b_i-\alpha)^2}{4a_i}  \right]
			 + \sum_{\left\{i\mid 2a_iC_i+b_i\le\alpha\right\}}\left(a_iC_i^2+b_iC_i\right)$$
\item No prosumer participation (strategic bidding): \begin{align*}\Wcal^{SN} &=\sum_{i}\Bigg[\frac{1}{4a_i}\left(\frac{\sum_{i}\left[\frac{\alpha N +b_i}{2a_i}+C_i \right]\cdot \mathds{1}_{\left\{y^{SN}> y^i\right\}} }{(N+1)\sum_{i}\frac{1}{2a_i}\cdot\mathds{1}_{\left\{y^{SN}>y^i\right\}}}\right)^2-\frac{b_i^2}{4a_i}\Bigg]\cdot\mathds{1}_{\left\{y^{SN}>y^i\right\}}  \\
&\qquad+\sum_{i}\left(a_iC_i^2 + b_iC_i\right)\cdot\mathds{1}_{\left\{y^{SN}\le y^i\right\}} +\alpha \frac{ N\sum_{i}\left(C_i - \frac{\alpha - b_i}{2a_i}\right)\cdot \mathds{1}_{\left\{y^{SN}> y^i\right\}}}{N+1}\end{align*}
\end{itemize}
\end{mdframed}
\caption{Explicit characterization of equilibrium social welfare under different models}
\label{fig:sw}
\end{figure*}
\normalsize

	Combining~\eqref{eq:disn} with~$\lambda= \lambda^{SN}$, Lemma~\ref{lem:yi}, and the supply-demand balance, we obtain that
	\small
	\begin{align}\label{eq:lambsn}
		\lambda^{SN}(y) = \frac{y+\sum_i\left(C_i+\frac{b_i}{2a_i}\right)\cdot \mathds{1}_{\left\{y>y^i\right\}}}{\sum_i\frac{1}{2a_i}\cdot\mathds{1}_{\left\{y>y^i\right\}}}.
	\end{align}
\normalsize
	
	Therefore, each generator essentially solves~\eqref{eq:genpsn} with the market price function given by~\eqref{eq:lambsn}. Since they are all identical, we only look at the equilibrium where $y_j = y_{j'},\forall j,j'\in[N]$. Thus, from the first-order condition of~\eqref{eq:genpsn}, we obtain the condition on optimal supply amount for generator~$j$:
	\small
	\begin{align}\label{eq:yjsn}
		y_j^{SN} = \frac{\sum_i\left(\frac{\alpha - b_i}{2a_i}-C_i\right)\cdot\mathds{1}_{\left\{Ny_j^{SN}>y^i\right\}}}{N+1},\quad \forall j.
	\end{align}
\normalsize
	In order to sell~$y_j^{SN}$ as given in~\eqref{eq:yjsn}, generator~$j$ will bid~$\tilde{c}_j$ such that the system operator will allocate~$y_j^{SN}$ amount of energy to~$j$. The generator thus considers the economic dispatch problem that the system operator solves.
	Therefore, the generator will bid~$\tilde{c}_j$ such that~$y^{SN}:= \sum_jy_j^{SN} = Ny_j^{SN}$ solves the system operator's problem optimally, i.e.,
	\small
	\begin{align}\label{eq:bidsconn}
		\frac{\partial u(C+y)}{\partial y}\bigg|_{y=Ny_j^{SN}} = \frac{\partial \tilde{c}_j(y_j)}{\partial y_j}\bigg|_{y_j = y_j^{SN}}.
	\end{align}
\normalsize
	While there are many possible choices of~$\tilde{c}_j$ that satisfies~\eqref{eq:bidsconn}, one of the optimal bidding strategies for the generator is again to bid a linear cost function.
	\begin{lemma}\label{lem:bidsn}
		The following linear cost function is an optimal bidding strategy for  generator~$j$:
		\small
		\begin{align}\label{bidsn}
			\tilde{c}_j(y_j) &=  \Bigg[\frac{N\alpha\sum_i\frac{1}{2a_i}\cdot\mathds{1}_{\left\{Ny_j^{SN}>y^i\right\}}}{(N+1)\sum_i\frac{1}{2a_i}\cdot \mathds{1}_{\left\{Ny_j^{SN}>y^i\right\}}}\nonumber\\
			&\qquad +\frac{\sum_i\left(C_i+\frac{b_i}{2a_i}\right)\cdot \mathds{1}_{\left\{Ny_j^{SN}>y^i\right\}}}{(N+1)\sum_i\frac{1}{2a_i}\cdot \mathds{1}_{\left\{Ny_j^{SN}>y^i\right\}}}\Bigg]\cdot y_j
		\end{align}
	\normalsize
	\end{lemma}

	Considering the prosumers' and the system operator's problems, we have the following result.
	\begin{proposition}[Competitive Equilibrium]\label{prop:power2n}
		Under Assumption~\ref{assum:power}, if all generators bid as in~\eqref{bidsn}, an optimal solution to~\eqref{eq:sopns} satisfies the following:
			\begin{subequations}
					\begin{align}
							z_i^{SN} &=  \left[- C_i + \frac{\lambda^{SN} - b_i}{2a_i}\right]^+,\quad \forall i,\label{eq:zsysn}\\
							y_j^{SN} &= \frac{\sum_i\left(\frac{\alpha - b_i}{2a_i}-C_i\right)\cdot\mathds{1}_{\left\{Ny_j^{SN}>y^i\right\}}}{N+1},\quad \forall j,\label{eq:zsysysn}
						\end{align}
				\end{subequations}
		where the equilibrium market price (the optimal Lagrange multiplier of the constraint) is
		$$
		\lambda^{SN}
		=\frac{N\alpha}{N+1} + \frac{\sum_i\left(C_i+\frac{b_i}{2a_i}\right)\cdot \mathds{1}_{\left\{y^{SN}>y^i\right\}}}{(N+1)\sum_i\frac{1}{2a_i}\cdot \mathds{1}_{\left\{y^{SN}>y^i\right\}}}.
		$$
		Furthermore, $z_i^{SN}$ and $y_i^{SN}$ optimally solve the prosumer's/generator's problem.
%

	\end{proposition}

	\section{DISCUSSION}\label{sec:dis}
	In previous sections, we have analyzed and derived the equilibrium quantities for all four models, depending on whether we allow prosumer participation and whether generators can bid strategically. For each model, Fig.~\ref{fig:sw} provides an explicit characterization of the equilibrium social welfare.

%
%
	While we have obtained explicit expressions (in terms of parameters of the prosumers and generators) for all four social welfare models, it remains to compare them to finish the proof of Theorem~\ref{thm:power}, which is not obvious in their current form. In this section, we provide alternative expressions for these social welfares. To proceed, we first define the social welfare when there is no generator and no electricity market, i.e., each prosumer consumes the amount of energy that her capacity allows:
	$
	\Wcal_0 := \sum_i u_i(C_i).
	$
	We also define the amount of market price rise due to strategic bidding of the generators as $\delta := \lambda^{S} - \lambda^T$, and~$\delta^N:=\lambda^{SN} - \lambda^{TN}$. With these definitions, we have the following expressions for the social welfare under different models, which use~$\Wcal_0$ as a reference.
	
		\begin{figure*}[ht]
\centering
\small
\begin{mdframed}
	The following relations hold:
	\begin{subequations}
		\begin{align}
			&\Wcal^T - \Wcal^{TN} = - \sum_{\left\{i\mid z_i^T \le 0\right\}}a_i\left(z_i^T\right)^2 \ge 0\label{eq:TTN}\\
			&\Wcal^S - \Wcal^{SN} = \sum_i\left(-a_i{\left(z_i^S\right)}^2 + \delta z_i^S\right) + \sum_ia_i\left(z_i^{SN}\right)^2\- \delta^N \sum_iz_i^{SN}\ge 0\label{eq:SSN}\\
			&\Wcal^T - \Wcal^S = \left(\frac{\sum_iz_i^T}{N+1}\right)^2\cdot \frac{1}{\sum_i\frac{1}{-a_i}} \ge 0\label{eq:TS}\\
			&\Wcal^{TN} - \Wcal^{SN} = \sum_{\left\{i\mid d_i^{SN} = 0, z_i^T>0\right\}}(-a_i)z_i^2 + \left(\frac{\sum_{\left\{i\mid d_i^{SN}>0\right\}}z_i^T}{N+1}\right)^2\cdot \frac{1}{\sum_{\left\{i\mid d_i^{SN}>0\right\}}\frac{1}{-a_i}}\ge 0\label{eq:TNSN}\\
			&\Wcal^{TN} - \Wcal^{SN}\ge \Wcal^T -\Wcal^S\label{eq:final}
		\end{align}
	\end{subequations}
\end{mdframed}
\caption{Restatement of Theorem~\ref{thm:power} (Main Result)}
\label{fig:Thm1}
\end{figure*}
\normalsize

\begin{figure*}
	\centering
	\includegraphics[width=0.9\linewidth]{numerical3.pdf}
	\caption{{\bf Left:} Social welfare for each market setup vs. $N$. DER participation improves the social welfare. As $N$ increases, strategic bidding converges to truthful bidding, and all inequalities provided in Theorem~\ref{thm:power} hold. {\bf Middle:} Quantification of market power in terms of social welfare loss. When DERs are integrated, market power is mitigated. {\bf Right:} Price for each market setup vs. $N$. Highest price corresponds to strategic bidding without DER participation, but when DERs are present, the price becomes lower. All prices converge to $\alpha=5$. }
	\label{numerical3}
\end{figure*}

	\begin{proposition}\label{prop:usingsw0}
		We have the following expressions for the four social welfares.
		\begin{align}
			\Wcal^T &= \Wcal_0 -\sum_ia_i{\left(z_i^T\right)}^2,\label{eq:WT}\\
			\Wcal^S &= \Wcal_0 + \sum_i\left(-a_i{\left(z_i^S\right)}^2 + \delta z_i^S\right),\label{eq:WS}\\
			\Wcal^{TN} &= \Wcal_0 - \sum_{i}a_i\left(z_i^{TN}\right)^2= \Wcal_0 - \sum_{\left\{i\mid z_i^T > 0\right\}}a_i\left(z_i^T\right)^2,\label{eq:WTN}\\
			\Wcal^{SN} &= \Wcal_0 - \sum_ia_i\left(z_i^{SN}\right)^2 + \delta^N \sum_iz_i^{SN}.\label{eq:WSN}
		\end{align}
	\end{proposition}
We next obtain the expressions for~$\delta$ and~$\delta^N$. Under full prosumer participation and strategic bidding of the generators, we have that
$
u'(C_i+z_i^S) = \lambda^S = \lambda^T + \delta = \alpha + \delta.
$
From the truthful bidding case, we also have that $\alpha = u_i'(C_i+z_i^T)$. Thus, we can write~$\delta$ as
$
\delta = u'(C_i+z_i^S) - u'(C_i+z_i^T) = 2a_i(z_i^S-z_i^T), \forall i,
$
or equivalently, by noting that~$\sum_iz_i^S = y^S = y^T\cdot\frac{N}{N+1} = \frac{N}{N+1}\sum_iz_i^T$, we have that
\begin{align}\label{eq:delta}
	\delta = \frac{\sum_i\left(z_i^S- z_i^T\right)}{\sum_i\frac{1}{2a_i}} = \frac{-\left(\frac{1}{N+1}\right)\sum_iz_i^T}{\sum_i\frac{1}{2a_i}}.
\end{align}
Under the models that prosumers cannot sell, with generators bidding strategically, we have that $u'_i(C_i+z_i^{SN}) = \lambda^{SN} = \lambda^{TN}+\delta^N = \alpha + \delta^N$ for those prosumers with~$z_i^{SN}>0$. From the truthful bidding case, we also have that~$\alpha = u_i'(C_i+z_i^{TN})$ for those prosumers with~$z_i^{TN}>0$. Thus, we can write~$\delta^N$ as
$
\delta^N = u_i'(C_i+z_i^{SN}) - u_i'(C_i+z_i^{TN}) = 2a_i(z_i^{SN} - z_i^{TN
}) = 2a_i(z_i^{SN} - z_i^{T}),\ \forall i\text{ s.t. }z_i^{SN}>0.
$
By noting that $\sum_{\left\{i\mid z_i^{SN}>0\right\}}z_i^{SN} = y^{SN} = \frac{N}{N+1}\sum_{\left\{i\mid z_i^{SN}>0\right\}}z_i^T$, we have that
\begin{align}\label{eq:deltaN}
	\delta^N = \frac{\sum_{\left\{i\mid z_i^{SN}>0\right\}}(-z_i^{SN}+z_i^T)}{-\sum_{\left\{i\mid z_i^{SN}>0\right\}}\frac{1}{2a_i}}= \frac{\frac{1}{N+1}\sum_{\left\{i\mid z_i^{SN}>0\right\}}z_i^T}{-\sum_{\left\{i\mid z_i^{SN}>0\right\}}\frac{1}{2a_i}}.
\end{align}

With Proposition~\ref{prop:usingsw0},~\eqref{eq:delta}, and~\eqref{eq:deltaN}, we are ready to show the relations in Theorem~\ref{thm:power}. After some algebra, we restate Thoerem~\ref{thm:power} as in Fig.~\ref{fig:Thm1}, which
finishes the proof of Theorem~\ref{thm:power}, and verifies the benefit of aggregating distributed energy resources (DERs), i.e., by allowing the aggregation of DERs, the optimal social welfare is improved with either truthful bidding generators or strategic bidding generators, and the loss of social welfare due to strategic bidding of the generators is reduced in the full participation model compared to that in the no participation model. Thus, we can state the main takeaway point of the paper, which is, {\em efficient DER aggregation mitigates market power of generators, and such mitigation is quantified explicitly as in Fig.~\ref{fig:Thm1}}. 

\section{ILLUSTRATIVE EXAMPLE}\label{sec:ex}
Next, we consider an example with $n=2$ and provide illustrations of our results. 
We let the true cost of each generator be $c(y)=\alpha y = 5 y$. For each prosumer $i$, we let $a_i=-0.1$ and  $b_i=10$, and we distinguish between them via the capacities, with $C_1=10$ and $C_2=30$. The parameters are picked such that $2a_1C_1+b_1>\alpha$  and $2a_2C_2+b_2<\alpha$, i.e., prosumer $i=1$ will always have a positive demand. To make our example more interesting, we vary the number of generators from $N=1$ to $N=20$ (outcomes saturate for $N>20$). In Fig.~\ref{numerical3}, we plot the social welfare for each market setup, efficiency loss due to strategic behavior of the generators, and prices. The key outcomes can be summarized as follows: 
\begin{enumerate}
\item $\Wcal^{SN} \rightarrow \Wcal^{TN}$ and  $\Wcal^{S} \rightarrow \Wcal^{T}$ as $N\rightarrow\infty$. 
\item $\lambda^{SN} \rightarrow \alpha$ and  $\lambda^{S} \rightarrow \alpha$  as $N\rightarrow\infty$. 
\item For $N<\infty$, $\alpha < \lambda^{SN} < \lambda^{S}$. 
\item Efficient DER aggregation mitigates the market power of generators. 
\end{enumerate}
\section{CONCLUSIONS AND FUTURE WORKS}\label{sec:con}

There are many open questions related to DER aggregation, as they are intermittent, uncontrollable, and rapidly increasing. While the work in \cite{9683117} provides an efficient DER aggregation model, market power of generators was not addressed. Consequently, this work serves as a stepping stone towards understanding the impact of DER aggregation onto market power of generators.  Particularly, we have proven the inequalities stated in Theorem \ref{thm:power} and provided explicit expressions later in Figure~\ref{fig:Thm1}. The main results of this paper state that {\it the optimal social welfare under efficient DER integration is always greater than that without prosumer participation, the optimal social welfare under truthful bidding is always greater than that under strategic bidding of generators, and the loss of social welfare due to strategic bidding of the generators is mitigated under efficient DER aggregation.}

DERs are naturally stochastic, so, incorporating uncertainty and its impact on market power mitigation would be an interesting question to address. Also, in our analysis, we have assumed quadratic utilities and linear generation costs, so, generalizations to generic concave utility functions and convex generation costs is another research direction. Finally, we have assumed identical generators, so it would also be interesting to explore generators of different kinds, each with generator-specific considerations, while taking into account the competition among them and network considerations.

\bibliographystyle{IEEEtran}
\bibliography{references}

\section*{APPENDIX}

\begin{proof}[Proof of Proposition~\ref{prop:power1}]
	We write the Lagrangian of~\eqref{eq:sop} as
	$
		\mathcal{L} = \sum_iu_i(C_i+z_i) - \sum_jc_j(y_j) +\lambda\left(\sum_jy_j-\sum_iz_i\right) + \mu_i(z_i+C_i) + \nu_jy_j,
	$
	where~$\lambda,\mu_i,\nu_j$ are the Lagrange multipliers of the constraints. The KKT optimality conditions are
	\small
	\begin{subequations}\label{eq:wkkt}
		\begin{align}
			&\frac{\partial \Lcal}{\partial z_i} = 2a_i(C_i+z_i)+b_i-\lambda +\mu_i = 0,\quad \forall i,\\
			&\frac{\partial \Lcal}{\partial y_j} = -\alpha +\lambda +\nu_j = 0,\quad \forall j,\\
			&\lambda\left(\sum_jy_j-\sum_iz_i\right) = 0,\quad \mu_i(z_i+C_i)  = 0,\ \forall i,\\
			&\nu_jy_j = 0,\ \forall j,\quad \sum_jy_j-\sum_iz_i = 0,\\
			& z_i > -C_i,\ \forall i,\quad y_j\ge 0, \ \forall j,\quad  \lambda, \mu_i,\nu_j \ge 0,\ \forall i,j.
		\end{align}
	\end{subequations}
\normalsize
	By Assumption~\ref{assum:power}, we have that $y_j>0$ and thus~$\nu_j =0$ for all~$j\in[N]$. With some algebra, we can conclude from~\eqref{eq:wkkt} that
	$
		\lambda^T = \alpha,\ z_i^T = -C_i+\frac{\alpha-b_i}{2a_i},\ \sum_jy_j^T = \sum_iz_i^T = -C+\sum_i\frac{\alpha-b_i}{2a_i}.
	$
	We note that the optimal~$z_i^T$ from the above is also optimal for the prosumer's problem. The optimal social welfare is thus
	$
		\Wcal^T = \sum_iu_i(C_i+z_i^T) - \sum_jc_j(y_j^T) = \alpha C-\sum_i\frac{(b_i-\alpha)^2}{4a_i}.
	$
	This completes the proof of Proposition~\ref{prop:power1}.
\end{proof}

\begin{proof}[Proof of Lemma~\ref{lem:bids}]
	Recall from~\eqref{eq:u} that $u(C+y) = \left\{\max_{z_i>-C_i}\sum_iu_i(C_i+z_i)\ \text{s.t.}\ \sum_iz_i  = y\right\},$
	which is itself an optimization problem. We write its Lagrangian:
	$
		\Lcal = \sum_iu_i(C_i+z_i) + \lambda\left(y-\sum_iz_i\right) + \mu_i(z_i+C_i),
	$
	where~$\lambda$ and~$\mu_i$ are the Lagrange multipliers of the constraints. The KKT optimality conditions are
	\begin{subequations}
		\begin{align}
			\frac{\partial \Lcal}{\partial z_i} &= 2a_i(C_i+z_i)+b_i-\lambda +\mu_i = 0,\quad \forall i,\\
			&\lambda\left(y - \sum_iz_i\right) = 0,\quad \mu_i(z_i+C_i) = 0,\ \forall i,\\
			&\sum_iz_i=y,\quad z_i> -C_i,\ \forall i,\quad \lambda,\mu_i\ge 0, \ \forall i.
		\end{align}
	\end{subequations}
	Thus, the optimal $z_i = \frac{\lambda - b_i}{2a_i}-C_i$, and thus
	$
	\sum_iz_i = \sum_i\frac{\lambda-b_i}{2a_i}-\sum_iC_i = y,
	$
	which implies that
	$
	\lambda = \frac{y+\sum_iC_i+\sum_i\frac{b_i}{2a_i}}{\sum_i\frac{1}{2a_i}}.
	$
	Therefore,
	$
		u(C+y) = \sum_{i}a_i\left(\frac{\lambda - b_i}{2a_i}\right)^2 + \sum_i b_i\left(\frac{\lambda - b_i}{2a_i}\right)= \sum_i \frac{1}{4a_i}\left(\frac{y+C + \sum_i\frac{b_i}{2a_i}}{\sum_i\frac{1}{2a_i}} - b_i\right)^2 + \sum_i \frac{b_i}{2a_i}\left(\frac{y+C + \sum_i\frac{b_i}{2a_i}}{\sum_i\frac{1}{2a_i}}-b_i\right),
	$
	which, together with~\eqref{eq:yjs}, implies that
	$
		\frac{\partial u(C+y)}{\partial y}\bigg|_{y=Ny_j^S} = \frac{\frac{- C + \sum_i\frac{\alpha - b_i}{2a_i}}{N+1}\cdot N+C + \sum_i\frac{b_i}{2a_i}}{\sum_i\frac{1}{2a_i}} \\
		= \frac{N\alpha\sum_i\frac{1}{2a_i} + C +\sum_i\frac{b_i}{2a_i}}{(N+1)\sum_i\frac{1}{2a_i}}.
	$
	Therefore, bidding~\eqref{eq:bids} ensures that the condition~\eqref{eq:bidscon} is satisfied, which ensures that the system operator will optimally assign~$y_j^S$ to generator~$j$.
\end{proof}

\begin{proof}[Proof of Proposition~\ref{prop:power2}]
	We write the Lagrangian of~\eqref{eq:sops} as
	$
		\mathcal{L} = \sum_iu_i(C_i+z_i) - \sum_j\tilde{c}_j(y_j) +\lambda\left(\sum_jy_j-\sum_iz_i\right) + \mu_i(z_i+C_i) + \nu_jy_j,
	$
	where~$\lambda,\mu_i,\nu_j$ are the Lagrange multipliers of the constraints, and~$\tilde{c}_j$ is given in~\eqref{eq:bids}. The KKT optimality conditions are
	\small
	\begin{subequations}\label{eq:wkkts}
		\begin{align}
			&\frac{\partial \Lcal}{\partial z_i} = 2a_i(C_i+z_i)+b_i-\lambda +\mu_i = 0,\ \forall i,\\
			&\frac{\partial \Lcal}{\partial y_j} = \frac{N\alpha\sum_i\frac{1}{2a_i} + C +\sum_i\frac{b_i}{2a_i}}{(N+1)\sum_i\frac{1}{2a_i}} +\lambda +\nu_j = 0,\ \forall j,\\
			&\lambda\left(\sum_jy_j-\sum_iz_i\right) = 0,\quad\mu_i(z_i+C_i)  = 0,\ \forall i,\\
			&\nu_jy_j = 0,\ \forall j,\quad \sum_jy_j-\sum_iz_i = 0,\quad z_i > -C_i,\ \forall i,\\
			& y_j\ge 0, \ \forall j,\quad \lambda, \mu_i,\nu_j \ge 0,\ \forall i,j.
		\end{align}
	\end{subequations}
\normalsize
	By Assumption~\ref{assum:power}, we have that $y_j>0$ and thus~$\nu_j =0$ for all~$j\in[N]$. With some algebra, we can conclude from~\eqref{eq:wkkts} that
	$
		\lambda^S =  \frac{N\alpha\sum_i\frac{1}{2a_i} + C +\sum_i\frac{b_i}{2a_i}}{(N+1)\sum_i\frac{1}{2a_i}},\
		z_i^S = -C_i+\frac{\lambda^S-b_i}{2a_i},\
		y_j^S = \frac{\sum_jy_j^S}{N} = \frac{\sum_iz_i^S}{N} = \frac{-C+\sum_i\frac{\lambda^S-b_i}{2a_i}}{N} = \frac{-C+\sum_i\frac{\alpha - b_i}{2a_i}}{N+1}.
	$
	We note that the optimal~$z_i^S$ we derived from the above also satisfies prosumer's optimal condition $2a_i(C_i+z_i^S) + b_i = \lambda^S$, and the optimal~$y_j^S$ from the above also satisfies~\eqref{eq:yjs}, i.e., $z_i^S$ and $y_j^S$ are optimal to prosumer~$i$ and generator~$j$, respectively. The optimal social welfare is given by
	$
		\Wcal^S = \sum_iu_i(C_i+z_i^S) - \sum_jc_j(y_j^S) = \left(\frac{{\sum_i\left[\frac{\alpha N +b_i}{2a_i}+C_i\right]}}{(N+1)\sum_i\frac{1}{2a_i}}\right)^2 \sum_i \frac{1}{4a_i} -\sum_i \frac{b_i^2}{4a_i} + \alpha \frac{N\sum_i\left(C_i-\frac{\alpha - b_i}{2a_i}\right)}{N+1}.
	$
	This completes Proposition~\ref{prop:power2}.
\end{proof}

\begin{proof}[Proof of Proposition~\ref{prop:power1n}]
	We write the Lagrangian of~\eqref{eq:sopn} as
	$
		\Lcal = \sum_iu_i(C_i+z_i) - \sum_jc_j(y_j) + \lambda\left(\sum_jy_j-\sum_iz_i\right) + \mu_iz_i + \nu_jy_j,
	$
	where~$\lambda, \mu_i,\nu_j$ are the Lagrange multipliers of the constraints. The KKT optimality conditions are
	\small
	\begin{subequations}\label{eq:wkktn}
		\begin{align}
			&\frac{\partial \Lcal}{\partial z_i} = 2a_i(C_i+z_i)+b_i-\lambda +\mu_i = 0,\ \forall i,\label{eq:kkta}\\
			&\frac{\partial \Lcal}{\partial y_j} = -\alpha +\lambda +\nu_j = 0,\ \forall j,\quad \lambda\left(\sum_jy_j-\sum_iz_i\right) = 0,\label{eq:kktc}\\
			&\mu_iz_i  = 0,\ \forall i,\quad\nu_jy_j = 0,\ \forall j,\quad \sum_jy_j-\sum_iz_i = 0,\label{eq:kktf}\\
			& z_i, y_j, \lambda, \mu_i,\nu_j \ge 0,\ \forall i,j. \label{eq:kktg}
		\end{align}
	\end{subequations}
\normalsize
	By Assumption~\eqref{assum:power}, we have that~$y_j>0$ and thus~$\nu_j = 0$ for all~$j\in[N]$ by~\eqref{eq:kktf}. From~\eqref{eq:kktc}, we have that~$\lambda^{TN} = \alpha$. Also~\eqref{eq:kkta},~\eqref{eq:kktf}, and~\eqref{eq:kktg} together imply that
	$
		z_i^{TN} = \left[-C_i+\frac{\alpha-b_i}{2a_i}\right]^+, \ \forall i.
	$
	From~\eqref{eq:kktc}, we then have that
	$
		\sum_iz_i^{TN} = \sum_i\left[-C_i+\frac{\alpha-b_i}{2a_i}\right]^+ = \sum_jy_j^{TN} = N\cdot y_j^{TN},
	$
	which implies that
	$
		y_j^{TN} = \frac{\sum_i\left[-C_i+\frac{\alpha-b_i}{2a_i}\right]^+}{N},\ \forall j.
	$
	Therefore, we may write the social welfare as
	$
		\Wcal^{TN} = \sum_iu_i(C_i+z_i^{TN}) - \sum_jc_j(y_j^{TN})= \sum_{\left\{i\mid 2a_iC_i+b_i>\alpha\right\}} \left[\alpha C_i - \frac{(b_i-\alpha)^2}{4a_i}  \right]+ \sum_{\left\{i\mid 2a_iC_i+b_i\le\alpha\right\}}\left(a_iC_i^2+b_iC_i\right).
	$
\end{proof}

\begin{proof}[Proof of Lemma~\ref{lem:ucondiff}]
	Recall from~\eqref{eq:uc+yn} that
	$u(C+y) = \left\{\max_{z_i\ge 0}\sum_iu_i(C_i+z_i)\ \text{s.t.}\ \sum_iz_i  = y\right\},$
	which is itself an optimization problem. We write its Lagrangian:
	$
		\Lcal = \sum_iu_i(C_i+z_i) + \lambda\left(y-\sum_iz_i\right) + \mu_iz_i,
	$
	where~$\lambda$ and~$\mu_i$ are the Lagrange multipliers of the constraints. The KKT optimality conditions are
	\small
	\begin{subequations}
		\begin{align}
			&\frac{\partial \Lcal}{\partial z_i} = 2a_i(C_i+z_i)+b_i-\lambda +\mu_i = 0,\ \forall i,\\
			&\lambda\left(y - \sum_iz_i\right) = 0,\quad \mu_iz_i = 0,\ \forall i,\quad\sum_iz_i=y\\
			&z_i, \lambda,\mu_i\ge 0, \ \forall i.
		\end{align}
	\end{subequations}
\normalsize
	The optimal $z_i$ is given by
	$
		z_i = \frac{\lambda-\mu_i-b_i}{2a_i} - C_i.
	$
	Since~$\mu_iz_i = 0$, for any given~$\lambda$, we have the set of prosumers with~$z_i>0$, i.e.,
	$
		\mathcal{S}(\lambda) = \left\{i\mid z_i>0\right\} = \left\{i\ \Big| \ \frac{\lambda-b_i}{2a_i}-C_i > 0\right\} = \left\{i\mid \lambda < 2a_iC_i+b_i\right\}.
	$
	Similarly, any prosumers in $\mathcal{S}^c(\lambda) := \left\{i\mid \lambda \ge 2a_iC_i+b_i\right\}$ will have~$z_i = 0$.
	Thus, we have that
	$
		\sum_iz_i = \sum_{i\in\mathcal{S}(\lambda)}\left[\frac{\lambda-b_i}{2a_i} - C_i\right] = y,
	$
	which implies that
	\small
	\begin{align}\label{eq:lambday}
		\lambda &= \frac{y + \sum_{i\in\mathcal{S}(\lambda)}C_i + \sum_{i\in\mathcal{S}(\lambda)}\frac{b_i}{2a_i}}{\sum_{i\in\mathcal{S}(\lambda)}\frac{1}{2a_i}}.
	\end{align}
\normalsize
	We will figure out an expression of the set~$\mathcal{S}$ as a function of~$y$.
	Recall that prosumers are sorted in decreasing order of~$2a_iC_i+b_i$. If~$\lambda\ge 2a_1C_1 + b_1$, then all prosumers have~$z_i = 0$, and~$\mathcal{S}(\lambda)$ is empty. As $\lambda$ decreases to~$2a_1C_1+b_1$, the first prosumer is included in the set.
	
	When the set~$\mathcal{S}(\lambda)$ does not change, as~$y$ increases, $\lambda$ will decrease according to~\eqref{eq:lambday}. When $y$ increases to some critical point~$y^i$ that the prosumer~$i>1$ is just about to be included in the set~$\mathcal{S}$, we look at the corresponding $\lambda$ right before~$i$ is included:
	$
		\lambda_- = 2a_{i}C_{i}+b_{i} = \frac{y^i+\sum_{i'=1}^{i-1}\left(C_{i'}+\frac{b_{i'}}{2a_{i'}}\right)}{\sum_{i'=1}^{i-1}\frac{1}{2a_{i'}}},
	$
	which implies that
	$
		y^i= \left(2a_{i}C_{i}+b_{i}\right)\sum_{i'=1}^{i-1}\frac{1}{2a_{i'}}  - \sum_{i'=1}^{i-1}\left(C_{i'}+\frac{b_{i'}}{2a_{i'}}\right).
	$
	When prosumer~$i$ is just included in the set, we have that
	$
		\lambda_+ = \frac{y^i+\sum_{i'=1}^{i}\left(C_{i'}+\frac{b_{i'}}{2a_{i'}}\right)}{\sum_{i'=1}^{i}\frac{1}{2a_{i'}}}.
	$
	
	One can verify that~$\lambda_+ = \lambda_-$, which implies that as $y$ increases, $\lambda$ continuously decreases, even at those critical points when more prosumers are being added to the set~$\mathcal{S}$. Therefore, the set of prosumers with~$z_i>0$ can be expressed as
	$
		\mathcal{S}(y) = \left\{i\mid y > y^i\right\}.
	$
	
	We may thus write
	$
		u(C+y) = \sum_{i\in\mathcal{S}^c(y)}\left(a_iC_i^2+b_iC_i\right) + \sum_{i\in\mathcal{S}(y)}\left[a_i\left(\frac{\lambda-b_i}{2a_i}\right)^2 + b_i\left(\frac{\lambda - b_i}{2a_i}\right)\right]$,
		which then leads to 
		$
		u(C+y)= \sum_{i\in\mathcal{S}}\bigg[\frac{1}{4a_i}\left(\frac{y + \sum_{i\in\mathcal{S}} C_i + \sum_{i\in\mathcal{S}}\frac{b_i}{2a_i}}{\sum_{i\in\mathcal{S}}\frac{1}{2a_i}}-b_i\right)^2 + \frac{b_i}{2a_i}\left(\frac{y + \sum_{i\in\mathcal{S}} C_i + \sum_{i\in\mathcal{S}}\frac{b_i}{2a_i}}{\sum_{i\in\mathcal{S}}\frac{1}{2a_i}}-b_i\right)\bigg]  +\sum_{i\in\mathcal{S}^c}\left(a_iC_i^2 + b_iC_i\right).
	$
	When $y$ is within the range that~$\mathcal{S}$ does not change, we have that
	$
		\frac{\partial u(C+y)}{\partial y} =\sum_{i\in\mathcal{S}} \bigg[\frac{1}{2a_i}\left(\frac{y + \sum_{i\in\mathcal{S}} C_i + \sum_{i\in\mathcal{S}}\frac{b_i}{2a_i}}{\sum_{i\in\mathcal{S}}\frac{1}{2a_i}}-b_i\right)\frac{1}{\sum_{i\in\mathcal{S}}\frac{1}{2a_i}}+\frac{b_i}{2a_i}\frac{1}{\sum_{i\in\mathcal{S}}\frac{1}{2a_i}}\bigg]= \sum_{i\in\mathcal{S}}\frac{1}{2a_i}\frac{1}{\sum_{i\in\mathcal{S}}\frac{1}{2a_i}}\left(\frac{y + \sum_{i\in\mathcal{S}} C_i + \sum_{i\in\mathcal{S}}\frac{b_i}{2a_i}}{\sum_{i\in\mathcal{S}}\frac{1}{2a_i}}\right) \\= \frac{y + \sum_{i\in\mathcal{S}} C_i + \sum_{i\in\mathcal{S}}\frac{b_i}{2a_i}}{\sum_{i\in\mathcal{S}}\frac{1}{2a_i}}
		= \lambda.
	$
	Thus, we can conclude that, as $y$ increases, $\frac{\partial u(C+y)}{\partial y}$ continuously decreases, and the overall utility of consumption~$u(C+y)$ is continuous and differentiable in~$y$.
\end{proof}

\end{document}